\documentclass[12pt]{amsart}

\usepackage{enumerate}
\usepackage{graphicx, animate}
\usepackage{amsmath}
\usepackage{amssymb}
\usepackage{amsfonts}
\usepackage[dvipsnames]{xcolor}
\usepackage{hyperref}
\usepackage{tikz}
\usepackage{float}

\allowdisplaybreaks

 \newtheorem{theorem}{Theorem}[section]
 \newtheorem{corollary}[theorem]{Corollary}
 \newtheorem{lemma}[theorem]{Lemma}

\newtheorem{observation}[theorem]{Observation}
\theoremstyle{definition}
\newtheorem{definition}[theorem]{Definition}
\theoremstyle{remark}

\newtheorem{fact*}{Fact}
\newtheorem{problem}{Problem}

\DeclareMathOperator{\Topiary}{Topiary}
\DeclareMathOperator{\topiary}{topiary}

\newcommand\dd{\mathrm d}

\newcommand{\til}{\raise.17ex\hbox{$\scriptstyle\mathtt{\sim}$}}

\newcommand\beq{\begin{equation}}

\newcommand\eeq{\end{equation}}

\newcommand{\bbm}{\left[ \begin{smallmatrix}}
\newcommand{\ebm}{\end{smallmatrix} \right]}
\newcommand{\bpm}{\begin{pmatrix}}
\newcommand{\epm}{\end{pmatrix}}
\numberwithin{equation}{section}

\newlength{\Mheight}
\newlength{\cwidth}

\newcommand{\dfn}[1]{{\bf #1}\index{#1}}

\title[Topiarism]{Indices of quadratic programs over reproducing kernel Hilbert spaces for fun and profit}

\author[G. Hutinet]{
Geoffrey Hutinet
}
\email[Geoffrey Hutinet]{ghutinet@haverford.edu}

\author[J. E. Pascoe]{
J. E. Pascoe$^\dagger$
}
\address{Department of Mathematics\\
Drexel University\\
3141 Chestnut St \\
 Philadelphia, PA 19104}
\email[J. E. Pascoe]{james.eldred.pascoe@drexel.edu}
\thanks{$\dagger$ Partially supported by National Science Foundation DMS Analysis Grant 2319010. Thanks to Mathematisches Forschungsinstitut Oberwolfach workshops ``Real Algebraic Geometry with a View toward Koopman Operator Methods" and ``Non-commutative Function Theory and Free Probability."
Part of this research was performed while the author was visiting the Institute for Pure and Applied Mathematics (IPAM), which is supported by the National Science Foundation (Grant No. DMS-1925919).}

\date{\today}


\begin{document}

\begin{abstract}
We give an abstract perspective on quadratic programming with an eye toward long portfolio theory geared toward explaining sparsity via maximum principles.
Specifically, in optimal allocation problems, we see that support of an optimal distribution
lies in a variety intersect a kind of distinguished boundary of a compact subspace to be allocated over.
We demonstrate some of its intelligence by using it to solve mazes and interpret such behavior as the underlying
space trying to understand some hypothetical platonic index for which the capital asset pricing model holds.
\end{abstract}

\maketitle
    
\section{Introduction}
    \emph{The manucript contains animations which may appear static in feature-incomplete pdf viewers.}

    Stock market indices, such as the S\&P 500, Dow Jones Industrial Average, NASDAQ Composite index and the like, are supposed to in some way capture the overall health and behavior of the market.
    The so-called ``Bogleheads" (followers of the philosophy of Jack Bogle \cite{bogle1, bogle2}) often make stronger claims such as one cannot beat such and such an index (such as the S\&P 500) in the long run. The point being that such an allocation reflects the broad allocation 
    of the resources of society, and perhaps that, in principle, the growth in our real resources cannot be too outpaced within the market, if only because the outpacing portion becomes synonymous with the market
    as the residual parts of the market become impossible to resolve by comparison. (Such can be viewed as a sort of \emph{efficient market hypothesis}, see \cite{samuelson}.)

    On the other hand, many proponents of ``value" advocate for a narrow range of extremely well-researched long investments, with a general preference for determinism over speculative large payoffs. (Large speculative payoffs
    are often sought in the arena of ``growth" investing, e.g. \cite{alger}.)
    Advocates of such a philosophy 
    often cite \cite{grahamdodd} as foundational. Here an investment being long means that you actually own something, as opposed to investment based on more exotic financial products,
    such as shorts, options, pet insurance and other derivative products.

    We give some formal justification to the apparently contradictory reasoning--
    our invisible index theorem says that the optimal portfolio will approximate some broad index if possible, but because of geometric features of the space of 
    securities, the optimal portfolio may remain sparse. Moreover, any optimal portfolio behaves somewhat like an index, and all assets that are not used are worse than some capital asset pricing model of that index would require.
    Sparsity, in turn, suggests that investing fundamentally is \dfn{massively multiplayer}-- even if we remove the assets underlying ones optimal portfolio, good approximations to the invisible index may be available using wildly different securities which fit together to make a coherent portfolio is a superficially different way. 

    \subsection{Diversification}

    First, we will review the formal justification for diversification (essentially the arithmetic-geometric mean inequality, that parallel beats sequential, etc.) and the consequent modern portfolio theory.
    \subsubsection{One coin} \label{onecoin}
        Consider the following game: one is allowed to bet any amount of money on a single coin flip, if one wins they get back double their bet, otherwise, one gets two fifths their bet back, rounding down any fractional pennies.
        For example, without loss of generality, assume we always pick heads and always bet all of our money. Starting with one dollar we flip the coin ten times and flip the sequence
            $$HTTTTTHHHT$$
        where each $H$ represents a head and $T$ represents a tail.
        We can track how much money we would have in the following table.
        {\tiny \begin{center}
        \begin{tabular}{ c | c c c c c c c c c c c }
             & 0 & 1 & 2 & 3 & 4 & 5 & 6 & 7 & 8 & 9 & 10 \\\hline
             Flip & - & H & T & T & T & T & T & H & H & H & T \\ 
             Value & \$1.00  & \$2.00 & \$0.80 & \$0.32 & \$0.12 & \$0.04 & \$0.01 & \$0.02 & \$0.04 & \$0.08 & \$0.03
        \end{tabular}
        \end{center}}
        Clearly, on such a play we have gotten somewhat unlucky.

        To analyze the game, let us first summarize the possible outcomes of a single flip in a table.
            {\tiny \begin{center}
        \begin{tabular}{ c | c c}
              & H & T\\\hline
             Outcome & 2 & 2/5 
        \end{tabular}
        \end{center}}
        Each outcome has a likelyhood or probability of happening.
        {\tiny \begin{center}
        \begin{tabular}{ c | c c}
              & H & T\\\hline
             Chance & 50\% & 50\% 
        \end{tabular}
        \end{center}}

        On glance, the game seems good. On average, we make $\frac{2 + 2/5}{2} = \frac{6}{5}$ times our money for every bet made, that is the expected value of an individual round 
        where we bet one unit of value is to get an extra fifth of a unit back. Obviously a great deal.

        On the other hand, suppose again that we always go all in and the difference between the number of heads and tails is less than or equal to $10\%.$
        Let $100$ be the number of flips so far. At most we have $55$ heads, and at minimum $45$ tails. Under such rosy assumptions, we will have
                $$2^{55}\left(\frac{2}{5}\right)^{45}=\frac{4^{50}}{5^{45}}= \left(\frac{4}{5^{9/10}}\right)^{50}\approx 0.94^{50} \approx 0.045$$
        units of value for every one put in intitially. (Here the $2^{n+n/10}$ term comes from compounding the wins, and $\left(\frac{2}{5}\right)^{n-n/10}$ from compounding losses.)
        Thus, in the long run with such a strategy, we lose money. From expectation we might have expected $(6/5)^{100} \approx 82,817,974$ units of value. Incongruous.

        One sees parodies of ``growth" in the regime of games as those above-- some may even have speculated that every individual company, no matter how successful, is destined to die.
        The dedicated reader who simulates such a game may observe that it looks like it might have actually been ``growing" during some periods, but 
            also goes to $0.$ One should not follow ghosts whole-heartedly into the abyss-- those that embody the bull in the china shop.
    \subsubsection{Martingale gambling}
        Suppose there is a game where you can bet any amount and if you win, you double your bet, and on a loss you lose everything. No prior assumptions on the odds
        except that the chance of winning is positive and constant. The martingale strategy supposedly employed by gamblers was to double ones bet on a loss.
        Thus, as the potential outcome of the next game had double the potential upside, they would be accounting of all their lost prior bets with a potential to gain what was to 
        be gained on the first game itself.
        \subsubsection{Going all-in lets the house play a martingale}
        If one continually goes all in, we see the exponential increase in bets that occurs in a martingale strategy. Note that after some point successive bets
        will need to be made with loans.

        That is, companies and banks can play more aggressive, high-risk strategies due to their extreme capacity for leverage.
        Presumably, when one invests in something that itself has debt, although they may have no obligation to pay, they still accept indirect risk due to the prior claim of the bondholders.
        One gauges how successful their use of leverage has been within the context of their overall strategy. As large amounts of debt and leverage are already baked into assets,
        analysis becomes very delicate when one wants to use leverage on top of that.
    \subsubsection{Two coins}
        Consider again the following game: one is allowed to bet any amount of money on a single coin flip, if one wins they get back double their bet, otherwise, one gets two fifths their bet back, rounding down any fractional pennies.
        Now assume one can play two games at once, splitting your money between the coins however you like.
         For example, without loss of generality, assume we always pick heads and always bet all of our money.
        We can track how much money we would have in the following table.
        {\tiny \begin{center}
        \begin{tabular}{ c | c c c c c c c c c c c }
             & 0 & 1 & 2 & 3 & 4 & 5 & 6 & 7 & 8 & 9 & 10 \\\hline
             Flip & - & H & T & T & T & T & T & H & H & H & T \\ 
             Flip & - & H & T & T & H & T & T & H & T & H & H \\ 
             Value & \$1.00  & \$2.00 & \$0.80 & \$0.32 & \$0.38 & \$0.14 & \$0.04 & \$0.08 & \$0.04 & \$0.08 & \$0.04
        \end{tabular}
        \end{center}}
        We have added a second coin to our original data in Section \ref{onecoin} and thus inherited some of its bad luck. We have made a penny more though.
        
        To analyze the game, let us first summarize the possible outcomes of a single flip in a table.
            {\tiny \begin{center}
        \begin{tabular}{ c | c c}
             Outcome & H & T\\\hline
            H & 2 & 6/5  \\
            T & 6/5 & 2/5
        \end{tabular}
        \end{center}}
        Each outcome has a likelyhood or probability of happening.
        {\tiny \begin{center}
        \begin{tabular}{ c | c c}
            Probability & H & T\\\hline
            H & 25\% & 25\%  \\
            T & 25\% & 25\%
        \end{tabular}
        \end{center}}

        On glance, the game seems good but similar in ruin to the first game as we only got a penny more on the small simulation above. On average, we make $\frac{2 +6/5 + 6/5 + 2/5}{4} = \frac{6}{5}$ times our money for every bet made, that is the expected value of an individual round 
        where we bet one unit of value is to get an extra fifth of a unit back. Still obviously a great deal.

        On the other hand, suppose again that we always go all in.
        Let $100$ be the number of flips so far, assume we have flipped 27 double tails, 50 mixed heads and tails and 23 double heads, a somewhat average outcome. We will have
                $$2^{23}\left(\frac{2}{5}\right)^{27} \left(\frac{6}{5}\right)^{50} \approx 1.38$$
        units of value at the end.
    \subsubsection{Magic coins}
        Consider again the following game: one is allowed to bet any amount of money on a single coin flip, if one wins they get back double their bet, otherwise, one gets two fifths their bet back, rounding down any fractional pennies.
        Now assume one can play two games at once, splitting your money between the coins however you like. Moreover, these coins are enchanted so that they always flip opposite outcomes.
         For example, without loss of generality, assume we always pick heads and always bet all of our money. 

        We can track how much money we would have in the following table.
        {\tiny \begin{center}
        \begin{tabular}{ c | c c c c c c c c c c c }
             & 0 & 1 & 2 & 3 & 4 & 5 & 6 & 7 & 8 & 9 & 10 \\\hline
             Flip & - & H & T & T & T & T & T & H & H & H & T \\ 
             Flip & - & T & H & H & H & H & H & T & T & T & H \\ 
             Value & \$1.00  & \$1.20 & \$1.44 & \$1.72 & \$2.06 & \$2.47 & \$2.96 & \$3.55 & \$4.26 & \$5.11 & \$6.13
        \end{tabular}
        \end{center}}
        We took our original unlucky sequence from Section \ref{onecoin} and turned it into quite a good one.

        To analyze the game, let us first summarize the possible outcomes of a single flip in a table.
            {\tiny \begin{center}
        \begin{tabular}{ c | c c}
             Outcome & H & T\\\hline
            H & 2 & 6/5  \\
            T & 6/5 & 2/5
        \end{tabular}
        \end{center}}
        Each outcome has a likelyhood or probability of happening.
        {\tiny \begin{center}
        \begin{tabular}{ c | c c}
            Probability & H & T\\\hline
            H & 0\% & 50\%  \\
            T & 50\% & 0\%
        \end{tabular}
        \end{center}}
        The outcome in terms of our value is deterministic. We earn a fifth of our money each time we play.
        Thus, by combining two anti-correlated coins, we have created a riskless game.

        For general configurations of coins, see the theory of Kelly strategies \cite{kelly1956}.

    \subsection{Modern portolio theory}
        Given a vector of average returns $\psi_i$ with covariances $\rho_{ij}$
        modern portfolio theory suggests that we maximize
            $$\sum \omega_i \psi_i - \frac{1}{2}\sum \omega_i\omega_j \rho_{ij}$$ 
        where $\omega_i$ sum to $1.$
        If one assumes that the underlying securities grow as geometric Brownian motions,
        the maximum corresponds to the portfolio with maximum asymptotic median growth rate.
        Namely, geometric Brownian motion satisfies the stochastic differential equation:
            $$\frac{dS}{S} = m \dd t + \sigma dB_t.$$
        At time $t,$
            $$S = e^{(m-\sigma^2/2)t + \sigma B_t}$$
        where $B_t$ is a Brownian motion which grows like $\sqrt{t}.$
        Thus, to find the optimal portfolio for long term median growth, one sees the need to maximize $m - \sigma^2/2$ which is given explicitly by the formula above.
        
        The following principle explains the reward for diversification as a reward for risk management (or generation determinism [for lower bounds]):        
        \emph{In terms of long term median growth, given a basket of securities with approximately equal returns, one is stochastically paid for destroying variance at a rate on one variance destroyed gives one half unit of median return.}

        The \dfn{capital asset pricing model} states that
            $$\psi_i-r = \beta_i (\psi_M -r)$$
        where $\psi_M$ is the return of the market portfolio (often modelled by an index such as the S\&P 500) and $r$ represents the return on a risk-free asset (usually modelled by something like a treasury bill,)
        and $\beta_i$ is the covariance of the market portfolio with asset $i$ divided by the variance of the market portfolio. The supposed point being that to beat the market one needs to choose stocks with more market risk, in the sense that their covariance with the market portfolio is higher than the variance of the market portfolio.
        Fama-French type models add extra terms to better fit data-- in principle, these are factors affecting the return that cannot be captured by mean-variance analysis.
        See \cite{francis2013modern} for futher overview.

        The approach to capital asset pricing modeling does not take into account an important phenomenon-- that a portfolio must cohere. When one selects an outfit, some things do not work together-- if you need to add some
        particular accessory, one may need to remove other articles to have the ensemble make sense. (Where harmony may be thought of as approximating the magic coins above.)
        We rectify such by dealing with coherent fragments of the market directly, which will be our topiaric index theory.
    \subsection{The maximum principle}
        The maximum principle states that in a steady state of a diffusion type process, the maximum concentration must occur on the boundary.
        (That is, if there were a concentration maximum at some interior point, then it would have diffused to other points in the interior.)
        Such occur in general in complex variables and functional analysis, where 
        it is often called the distinguished or Shilov boundary. We shall define an appropriate frontier for which the maximum principle holds in quadratic programming.
    \subsection{Main results}
        We establish that optimal portfolios must be chosen from some potentially small set we call the green frontier in our \dfn{green topiary theorem}. We also show that if in some expanded universe there
        is a totally diversified index fund as the optimal portfolio, then finding the optimal portfolio in a restricted class of assets is merely trying to approximate the fluctuations of the 
        index, which we call the \dfn{invisible index theorem}. In general, we treat portfolios which are internally consistent and assets which are inconsistent with such must perform worse
        than the capital asset pricing model would predict,
        which we call the \dfn{capital asset pricing inequality}. 
        
        
\section{Preliminaries on Hilbert spaces} 
    We now recall some notions from the theory of Hilbert spaces.
   
    Let $\mathcal{H}$ be a real vector space.
    Call $\langle\cdot ,\cdot \rangle: \mathcal{H}\times \mathcal{H} \rightarrow \mathbb{R}$ an \dfn{inner product} if
        \begin{enumerate}
            \item $\langle x, x \rangle \geq 0,$  
            \item If $\langle x, x \rangle = 0,$ then $x=0,$
            \item $\langle x, y \rangle = \langle y, x \rangle,$
            \item $\langle \alpha x+\beta z, y \rangle = \alpha\langle x, y\rangle + \beta\langle z, y\rangle.$
        \end{enumerate}
    We call $\|x\| = \sqrt{\langle x, x \rangle}$ the \dfn{norm}.

    We a vector space with an inner product a \dfn{real Hilbert space} if whenever $\sum^{\infty}_{n=1} \|x_n\|$ converges
    then $\sum^{\infty}_{n=1} x_n$ converges.

    Examples include:
    \begin{enumerate}
        \item $\mathbb{R}^n$, where $\langle(x_1,\ldots,x_n), (y_1,\ldots,y_n)\rangle = \sum x_iy_i,$ (Euclidean space)
        \item If $X, Y$ are random variables with finite variance, $\langle X, Y\rangle = Cov(X,Y),$
        \item $\mathbb{R}^n$, $A$ a positive semidefinite matrix, where $\langle x, y\rangle = y^TAx,$ 
        \item The set of square summable sequences where $\langle(x_1, x_2, \ldots), (y_1, y_2, \ldots)\rangle = \sum x_iy_i,$
        called $\ell^2.$
        \item The set of functions on the unit interval such that $\int^1_0 |f|^2 \dd x < \infty.$
        where $\langle x,y \rangle = \int^1_0 fg \dd x.$ (Called $L^2.$)
    \end{enumerate}

    Say $x \perp y$ or $x$ is \dfn{perpendicular} to $y$ if $\langle x, y\rangle =0.$
    Hilbert spaces satisfy the Pythagorean theorem, where it is obtained via algebra in the following classic proof. 
   \begin{theorem}[The Pythagorean theorem]
        $x \perp y$ if and only if $$\|x\|^2+\|y\|^2 = \|x+y\|^2.$$
    \end{theorem}
    \begin{proof}
        \begin{align*}
            \|x\|^2+\|y\|^2 &= \langle x, x \rangle + \langle y, y \rangle \\
            &= \langle x, x \rangle + 2\langle x, y \rangle + \langle y, y \rangle \\
            &= \langle x+y, x+y \rangle = \|x+y\|^2
        \end{align*}
    \end{proof}
    Because of the relationship to coavariance, perpendicularity can help capture the notion of statistical independence.
    Another important fact is the following inequality.
        \begin{theorem}[Cauchy-Schwarz inequality]
            $\langle x, y\rangle \leq \|x\|\|y\|.$
        \end{theorem}

    The following will be very powerful for us.
    \begin{theorem}[The Riesz representation theorem]
        Let $\mathcal{H}$ be a Hilbert space.
        If $\lambda:\mathcal{H}\rightarrow \mathbb{R}$ is bounded and linear, that is $|\lambda(x)|\leq C\|x\|,$
        then there exists a unique $\tilde{\lambda}\in \mathcal{H}$ such that 
            $$\lambda(x) = \langle x, \tilde{\lambda}\rangle.$$
    \end{theorem}

    \begin{definition}
        A \dfn{real reproducing kernel Hilbert space} $\mathcal{H}$ is a Hilbert space of real-valued functions on some domain $\Omega$
        such that point evaluation at each $\alpha \in \Omega$ is a bounded linear functional. 
        
        We call the element realizing this via the Riesz representation theorem $k_\alpha.$ That is,
            $$f(\alpha)=\langle f, k_\alpha\rangle.$$
    \end{definition}
    See \cite{paulsen2016introduction} for reference.

    Examples include:
    \begin{enumerate}
        \item Euclidean space, where we view vectors as functions on $\{1,\ldots, n\},$
        \item $\ell^2,$
        \item NOT $L^2$ of $[0,1]$ (such functions are merely defined almost everywhere, that is point evaluations are not defined)
        \item The space of functions on the complex unit disk of the form
            $\sum^{\infty}_{-\infty} c_n z^n$
        where $\sum |c_n|^2 < \infty$ and $c_n = \overline{c_{-n}}.$
        Here,
                $$k_w(z) = \textrm{Re } \frac{1+z\overline{w}}{1-z\overline{w}}.$$
        Called the \dfn{real Hardy space} of harmonic functions on the disk.
    \end{enumerate}
    
    \begin{definition}
        Let $\mathcal{H}$ be real reproducing kernel Hilbert space on some domain $\Omega.$
        Let $\mu$ be a compactly supported distribution on $\Omega.$
        We associate the \dfn{embedded distribution} $\mu \in \mathcal{H}$
        via the Riesz representation theorem of the functional given by integration against $\mu.$
        That is,
            $$\lambda_\mu(f)=\int f \dd \mu = \langle f, \mu\rangle.$$
        (We abuse notation to treat these two objects the same.)
        
        We call the map taking a measure to its corresponding Hilbert space element the \dfn{kernel embedding of distributions}.

        We define the \dfn{space of embedded measures} to be the weak closure of the image of the kernel embedding of distributions.
    \end{definition}
    For a finite sum of point masses $\nu= \sum \alpha_i\delta_{x_i}$ we view $\nu = \sum \alpha_i k_{x_i}$ and $$\langle f,\nu\rangle=\sum \alpha_i f(x_i).$$

    See \cite{smolka} for information on kernel embeddings and learning applications. Our results explain sparsity (thus fast to evaluate) in the positive analog of support vector machines.
    (Instead of doing the job of separating two sets, instead one can think of this as wanting to find the outline of a set.)
\section{Topiarism}

    Our basic problem of study is as follows.

    \begin{definition}
        Let $\mathcal{H}$ be real reproducing kernel Hilbert space on some domain $\Omega.$
        Let $\psi$ be a continuous function on $\Omega.$
        Let $K \subseteq \Omega$ be compact.
        We define the \dfn{aesthetic objective} to be 
            $$\mathfrak{O}(\mu)=\int \psi \dd \mu - \|\mu\|^2/2.$$
        We call the maximizer of $\mathfrak{O}$ over all distributions \dfn{the topiary of $K$ with respect to $\psi$ over $\mathcal{H},$}
        denoted $\topiary(K).$
    \end{definition}
    The minimizer may not be unique as a measure, but the embedded measure is unique.
    When $\psi + r \in \mathcal{H}$ for some $r \in \mathbb{R},$ we have that maximizing the aesthetic objective is equivalent to
    minimizing $\|\psi + r - \mu \|.$
    
    That is,
        minimizing
            $$\|\psi -\mu \|^2 = \|\psi\|^2-2\langle \psi, \mu \rangle+\|\mu\|^2,$$
    is equivalent to maximizing
        $$\mathcal{O}(\mu)=\int \psi \dd \mu - \|\mu\|^2/2$$
    as $\psi$ is fixed.
    That is, from the closest measure perspective, the latter aesthetic formulation we have adopted need not require $\psi \in \mathcal{H}.$

    An important example is the problem of optimization of long portfolios (and other operations research problems) following
    Markowitz modern portfolio theory \cite{francis2013modern}. (The asymptotic growth rate of a geometric Brownian motion is equal to its mean minus its variance over two.)
    The name topiary is chosen in relation to the financial concept of hedging.

    \begin{lemma}
        Let $\mathcal{H}$ be real reproducing kernel Hilbert space on some domain $\Omega.$
        Let $\psi$ be a continuous function on $\Omega.$
        Let $K \subseteq \Omega$ be compact.
        The aesthetic objective satisfies
            $$D\mathfrak{O}(\mu)[\delta_x-\mu] = \psi(x)-\mu(x) - \int \psi(t)-\mu(t) \dd \mu(t).$$

        Thus, we call $$\iota_{\mu} = \psi - \mu - \int \psi(t)-\mu(t) \dd \mu(t)$$ the \dfn{aesthetic margin}.
        We define the score $\mathfrak{S}(\mu)$ of $\mu$ to be the supremum of the aesthetic margin.
    \end{lemma}

    \begin{theorem}
        Let $\mathcal{H}$ be real reproducing kernel Hilbert space on some domain $\Omega.$
        Let $\psi$ be a continuous function on $\Omega.$
        Let $K \subseteq \Omega$ be compact.
        The aesthetic margin of $\topiary(K)$ is $0$ on its support and nonpositive on $K.$
        We call $\int \psi(t)-\mu(t) \dd \mu(t)$ the \dfn{available topiaric rate}, denoted $r_K.$

        We define the \dfn{topiaric index} of $K$, denoted $\Topiary(K),$ to be the preimage of zero for the aesthetic margin
        of the topiary. That is, $$\Topiary(K) =\iota^{-1}_{\topiary(K)}(0).$$
        We call any set arising as the topiaric index of some $\mathcal{K}$ a \dfn{topiaric index}.
    \end{theorem}
    \begin{proof}
        The aesthetic margin must be be nonpositive at an optimum. If it were negative at a point its support, one would be able to scale up outside of a neighborhood of the point and scale down within a neighborhood of the point to to increase the aesthetic objective.
    \end{proof}
    Call a set of the form $\{x|\psi(x) - f(x) = C\}$ for some $f\in \mathcal{H}, C \in \mathbb{R}$ a \dfn{marginal hypersurface}.
    We see the following principle: \emph{A topiaric index for a set $K$ lies on a marginal hypersurface.}

    \section{The capital asset pricing inequality}

    The connection to nonpositive functions allows one to interface with Herglotz type theory as was done in \cite{jepkm} and replicate approaches to the Julia-Caratheodory theory \cite{ju20, car29, wol26}
    as in \cite{pascoe2021controlled, pascoePEMS} following the work of Agler, McCarthy and Young \cite{amy10a}, Agler, Tully-Doyle and Young \cite{aty12} and Agler and McCarthy \cite{amchvms}
    which give examples where regularity of a function forces the regularity of some underlying Agler kernels.
    In complex variables, we may also thus have connections to the theory of stable polynomials, their geometry,
    and the relation to complex function theory as in \cite{bps1, bps2, ams06, agmc_dv}.
    For a general introduction to advanced kernel methods and geometry beyond that of basic repoducing kernel Hilbert spaces, see \cite{ampi}.
    The point being that with an appropriate view, one should see a theory relating the behavior on the topiaric index of $K$ to behavior on the rest of $K.$
    A concrete problem would be to understand the Fama-French theory (or general arbitrage pricing theory, although the von Neumann elephant trunk-wiggling phenomenon there probably obfuscates the matter) in terms of the capital asset pricing inequality below.

    In the case of portfolio optimization, if there is a risk free asset in the optimal portfolio, the available topiaric rate is equal to the risk free rate.
    \begin{theorem}[Capital asset pricing inequality]
        Let $\mathcal{H}$ be real reproducing kernel Hilbert space on some domain $\Omega.$
        Let $\psi$ be a continuous function on $\Omega.$
        Let $K \subseteq \Omega$ be compact.
        Let $\mu\neq 0$ be the topiary of $K.$
        Define the \dfn{topiaric beta} of $x\in \Omega$ to be
            $$\beta(x) = \frac{\topiary(K)(x)}{\|\topiary(K)\|^2}.$$
        Note $\beta\in \mathcal{H}.$
        For every $x\in K$ we have that
            $$\psi(x) - r_K \leq \beta(x)\left(\int \psi \dd \mu -r_K\right)$$
        where equality holds exactly on the topiaric index of $K.$
    \end{theorem}
    \begin{proof}
        Let $\mu$ be the topiary of $K$.
        Note $\psi - \mu - \int \psi(t)-\mu(t) \dd \mu(t) \leq 0.$
        Also, $\|\mu\|^2 = \int \psi \dd \mu - r_K.$
        So, $\psi - r_K = \mu = \frac{\mu}{\|\mu\|^2}(\int \psi \dd \mu - r_K).$
    \end{proof}
    The classical capital asset pricing model \cite{francis2013modern} is essentially the assumption that
    whatever assets being analyzed form a topiaric index and contain a risk free asset. We note that ``alpha" based analysis with respect to a topiaric index returns alpha uniformly nonpositive,
    that is $$\alpha(x)=\psi(x) - r_K - \beta(x)\left(\int \psi \dd \mu -r_K\right)$$
    is nonpositive.

    The exact elaboration of Julia-Caratheodory-Wolff theory in such a setting is unclear, especially with respect to more sophisticated notions such as horocycles.
    As with the approach of Agler, McCarthy and Young of Julia-Caratheodory type theorems \cite{amy10a}, we derive estimates using Cauchy-Schwarz.
    A somewhat simplistic interpretation is that they describe regularity of functions near boundary optima.
    \begin{theorem}[Topiaric Julia-Caratheodory inequality]
        Let $\mathcal{H}$ be real reproducing kernel Hilbert space on some domain $\Omega.$
        Let $\psi$ be a continuous function on $\Omega.$
        Write $d(x,y) = \|k_x-k_y\|.$
        Let $K \subseteq \Omega$ be compact.
        Let $\mu$ be the topiary of $K.$
        Let $x$ be in the topiaric index of $K.$
        For every $y\in K$ such that $d(x,y)\neq 0$ we have that
            $$-\|\psi\| \leq \frac{\psi(y) - \psi(x)}{d(y,x)} \leq \frac{\mu(y)-\mu(x)}{d(y,x)}\leq \|\mu\|$$
        where $\|\psi\|$ if formally infinite if $\psi \notin \mathcal{H}.$
        Here, the middle inequality is an equality if $y$ is also in the topiaric index.
    \end{theorem}
    \begin{proof}
        Note,
            $$\psi(x) - r_K = \beta(x)\left(\int \psi \dd \mu -r_K\right) = \mu(x),$$
        and
            $$\psi(y) - r_K \leq \beta(y)\left(\int \psi \dd \mu -r_K\right) = \mu(y).$$

        Subtracting, 
            $$\psi(y) - \psi(x)  \leq \mu(y) -\mu(x).$$
        Thus, the middle inequality is satisfied.
        Now,
            $$\mu(y) -\mu(x) = \langle\mu , k_y- k_x \rangle \leq \|\mu\|\|k_y- k_x\|$$
        by the Cauchy-Schwartz inequality.
        Substituting $d(x,y) = \|k_x-k_y\|$ and rearranging gives the right hand inequality.
        Similarly, if so defined
            $$\psi(y) -\psi(x) = \langle\psi , k_y- k_x \rangle \geq -\|\psi\|\|k_y- k_x\|$$
        by Cauchy-Schwartz.
    \end{proof}
    Informally, we may interpret $\psi$ as more contractive that $\mu.$ 
    \begin{figure}[H]
    \centering
    \begin{tikzpicture}
        \fill[yellow, opacity=0.2] (-1,-1.1) -- (-.1,-1.1) -- plot[domain=0:5, samples=2] ( {(\x - 1 + 2)},\x) -- ( -1,5) -- cycle;

        \draw[<->] (0,-1.1) -- (0,5) node[right] {$\mu(y)$};
        \draw[<->] (-1,0) -- (5,0) node[above] {$\psi(y)$};
        
        \fill[black] (1,0) circle (1.5pt);        

        \fill[red] (2,1) circle (2pt);
        \fill[red] (1.5,.5) circle (2pt);

        \fill[red] (4,3) circle (2pt);

        \fill[red] (4.3,3.3) circle (2pt);

        \fill[brown] (3.4,3.4) circle (2pt);
        \fill[brown] (-.5,-.5) circle (2pt);
        \fill[brown] (.75,4.5) circle (2pt);
        \fill[violet] (.5,2.5) circle (2pt);

        \fill[violet] (4.5,-.5) circle (2pt);
         \fill[violet] (3.5,1) circle (2pt);
 \fill[violet] (4,2) circle (2pt);
        
        \node at (2, 1) [below right] {$(\mu(x), \psi(x))$};
        \node at (1, 0) [below right] {$r_K$};
        
        \draw[<->,domain=-1.1:5, dashed, samples=2] 
            plot ({(\x - 1 + 2)},\x) node[right] {};

    \end{tikzpicture}
    \caption{Our topiaric Julia-Caratheodory inequality says that the elements in the topiaric frontier lie on a line with slope $1$, depicted in red, elements of $K$, depicted in brown
    must be below the line, whereas elements in $\Omega$ not in $K$ may be above or below, depicted in violet. The line must intersect the $\psi(y)$ axis at the available topiaric rate $r_K.$
    Thus, the topiaric Julia-Caratheodory inequality is somewhat analogous to the classical security market line, \cite{francis2013modern}.}
    \label{fig:slope1line}
    \end{figure}
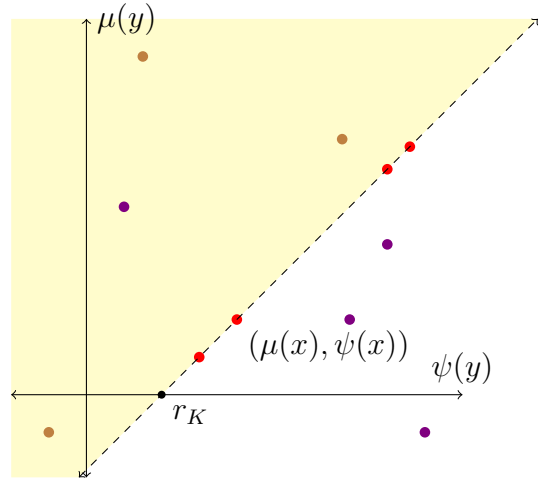

    \begin{figure}[H]
        \centering
        \begin{tikzpicture}
            \fill[yellow, opacity=0.2] (-1,-1.1) -- (-.1,-1.1) -- plot[domain=0:5, samples=2] ( {(\x - 1 + 2)},\x) -- ( -1,5) -- cycle;

            \draw[<->] (0,-1.1) -- (0,5) node[right] {$\mu(y)$};
            \draw[<->] (-1,0) -- (5,0) node[above] {$\psi(y)$};
            
            \fill[black] (1,0) circle (1.5pt);        

            \fill[red] (2,1) circle (2pt);
            \fill[red] (1.5,.5) circle (2pt);

            \fill[red] (4,3) circle (2pt);

            \fill[red] (4.3,3.3) circle (2pt);

            \fill[brown] (-1,-0.1) circle (2pt);
            \fill[brown] (-.5,0.2) circle (2pt);
            \fill[brown] (-.75,0.1) circle (2pt);
            \fill[brown] (.2,.3) circle (2pt);
            \fill[violet] (-.5,0) circle (2pt);

            \fill[violet] (4.5,-.5) circle (2pt);
             \fill[violet] (3.5,1) circle (2pt);
     \fill[violet] (4,2) circle (2pt);
            
            \node at (2, 1) [below right] {$(\mu(x), \psi(x))$};
            \node at (1, 0) [below right] {$r_K$};
            
            \draw[<->,domain=-1.1:5, dashed, samples=2] 
                plot ({(\x - 1 + 2)},\x) node[right] {};

        \end{tikzpicture}
        \caption{If instead we have that the remainder of $\Omega$ lies on the axis, we see that the topiary has a small inner product with the market, which in interpretations 
        as covariance gives that it is independent of the rest of the space. (One can imagine a variety of reasons for such in the securities context, innovation, new ideas, fraud, memes and so on.)}
        \label{fig:slope1line2}
    \end{figure}
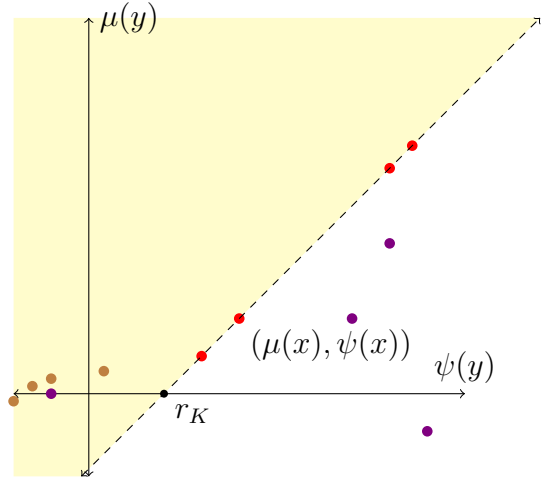

    \section{The invisible index theorem}

    We can become more comfortable with both the empirical failure of the
    capital asset pricing model and the obviously contradictory held-fast belief in such if we employ the belief that there is some, perhaps hypothetical, bigger ``platonic" \dfn{invisible index} which is supported everywhere out there
    where the model is saved. (Called such in analogy with the classical invisible hand.) Note, for example, no investor has access to all assets-- matters such as salience, sanctions and other barriers. See Roll's critique \cite{roll1977critique} on lack of testability of the classical capital asset pricing model.
    Note nothing obstructs the topiary of such a topiaric index minus a few key elements from being quite sparse. 
    \begin{theorem}[Invisible index theorem]
        Let $\mathcal{H}$ be real reproducing kernel Hilbert space on some domain $\Omega.$
        Let $\psi$ be a continuous function on $\Omega.$
        Let $K_1 \subseteq K_2 \subseteq \Omega$ be compact. Suppose $K_2$ is a topiaric index.
        The topiary of $K_1$ minimizes $\|\mu-\topiary(K_2)\|.$
    \end{theorem}
    \begin{proof}
        We see to minimize
       $\|\mu-\topiary(K_2)\|^2 = \|\mu\|^2 -2\int \topiary(K_2) \dd \mu + \|\topiary(K_2)\|^2.$
        Thus, we want to maximize  $\int \topiary(K_2) \dd \mu - \|\mu\|^2/2$
        As $K_2$ is an index $\topiary(K_2) = \psi + r$ for some $r\in \mathbb{R},$
        and thus we have our original maximization problem.
    \end{proof}

    To complement the above we note that, if $\psi\in \mathcal{H}$, then maximizing our aesthetic objective is equivalent to minimizing $\|\psi - \mu\|.$
    Namely, if $\psi$ is an embedded distribution, then $\psi$ is the topiary.

    \section{On the green frontier}
    We now define the appropriate distinguished boundary for our present consideration.
    \begin{definition}
        Let $\mathcal{H}$ be real reproducing kernel Hilbert space on some domain $\Omega.$
        Let $\psi$ be a continuous function on $\Omega.$
        Let $K \subseteq \Omega$ be compact.
        Let $\mathcal{U}$ be the collection of all open dense subsets of the space of embedded measures on $K.$
        Let 
            $$M_U = \{ x\in K | \iota_\mu(x)=\sup_K \iota_\mu \textrm{ for some }\mu \in U\}.$$
        We define the \dfn{green frontier} by the formula $$Green(K) = \bigcap_{U \in \mathcal{U}} \overline{M_U}.$$
    \end{definition}

    Note that given a reproducing kernel Hilbert space of harmonic functions with harmonic $\psi$
    that $Green(K) \subseteq \partial K$ by the maximum principle. Essentially, the green frontier is similar to the Shilov boundary in functional analysis, or the distinguished boundary in complex analysis.
    Moreover, via Krein-Milman, the green frontier is a subset of the extreme points of $K$ if we view $K$ as a subset of $\mathcal{H}$ via the kernel embedding.

    \begin{theorem}[Green topiary theorem]
        Let $\mathcal{H}$ be real reproducing kernel Hilbert space on some domain $\Omega.$
        Let $\psi$ be a continuous function on $\Omega.$
        Let $K \subseteq \Omega$ be compact.
        There exists a measure supported on the green frontier of $K$ which embeds to $\topiary(K).$
    \end{theorem}
    The green topiary theorem is an immediate consequence of the effectiveness of gradient ascent
    given by the update inequality, Theorem \ref{update}. 
    As the support must lie on a marginal hypersurface, we have that the support of the topiary is contained in the intersection of the green frontier with a marginal hypersurface--
    such is deeply related the the theme of \dfn{semi-stability}. In classical semi-stability regimes, we often have extremely constrained geometry and in particular low dimensionality
    \cite{bps1, bps2, ams06, knese2015integrability}.
    
\section{Approximation}
    We now give some perspectives on approximating the topiary, especially in light of empirical sparsity.
    \subsection{Ascent methods}
    First we expand on our discussion of gradient ascent type methods.
    \begin{lemma}
        Let $\mathcal{H}$ be real reproducing kernel Hilbert space on some domain $\Omega.$
        Let $\psi$ be a continuous function on $\Omega.$
        Let $K \subseteq \Omega$ be compact.
        We have that
        $$\mathfrak{S}(\mu) \geq \mathfrak{O}(\topiary(K)) - \mathfrak{O}(\mu) + \|\topiary(K)-\mu\|^2/2.$$
    \end{lemma}
    \begin{proof}
        Integrating the aesthetic margin with respect to the topiary gives exactly the right hand side,
        thus there is at least one $x$ with aesthetic margin that large.
    \end{proof}

    The following is a calculation.
    \begin{lemma}
        Let $\mathcal{H}$ be real reproducing kernel Hilbert space on some domain $\Omega.$
        Let $\psi$ be a continuous function on $\Omega.$
        Let $K \subseteq \Omega$ be compact.
        Let $\mu$ be a distribution on $K.$
        The optimal measure $\mu_t = (1-t) \mu + t\delta_x$ occurs at
            $$t = \frac{\iota_{\mu}(x)}{\|k_x-\mu\|^2}.$$
        Moreover,
            $$\mathfrak{O}(\mu_t) - \mathfrak{O}(\mu) = \frac{\iota_{\mu}(x)^2}{2\|k_x-\mu\|^2}.$$
    \end{lemma}

    Combining the two previous we see the following.
    \begin{theorem}[Update inequality] \label{update}
        Let $\mathcal{H}$ be real reproducing kernel Hilbert space on some domain $\Omega.$
        Let $\psi$ be a continuous function on $\Omega.$
        Let $K \subseteq \Omega$ be compact.
        Let $\mu$ be a distribution on $K.$
        Choose $x$ an optimum of the aesthetic margin.
        For the optimal measure $$\mu_t = (1-t) \mu + t\delta_x$$ we have that
            $$\mathfrak{O}(\mu_t) - \mathfrak{O}(\mu) \geq 
            \frac{[\mathfrak{O}(\topiary(K)) - \mathfrak{O}(\mu) + \|\topiary(K)-\mu\|^2/2]^2}
            {2\|k_x-\mu\|^2}.$$
    \end{theorem}

    The following follows immediately from the previous.
    \begin{theorem}[Topiaric law of large numbers]
        Let $\mathcal{H}$ be real reproducing kernel Hilbert space on some domain $\Omega.$
        Let $\psi$ be a continuous function on $\Omega.$
        Let $K \subseteq \Omega$ be compact.
        Optimally generate a sequence of measures $\mu_n$ as in the update inequality.
        The aesthetic objective of $\mu_n$ converges to that of $\topiary(K)$ with rate at worst $O(1/n).$
        The sequence $\mu_n$ converges at rate at worst $O(1 / n^{1/4}).$ If $\psi \in \mathcal{H},$ 
        then $\mu_n$ converges at rate at worst $O(1 / n^{1/2}).$
    \end{theorem}
    Compare to \cite{del1999central} which gives results on the dumptruck-sandpile-moving Wasserstein distance for measure approximation .

    \subsubsection{Greedy constructions}
    
The \dfn{greedy construction} of the topiary
$\textrm{greed}_n$ is constructed by taking the convex combination of $\textrm{greed}_{n-1}$  and a point mass at a maximizer contained in $\mathrm{Green}(A)$ of $\iota(\textrm{greed}_{n-1})$
such the aesthetic objective of that convex combination in maximized. Choosing $\textrm{greed}_0$ in $\mathrm{Green}(A)$ is advised.
The greedy construction converges to a measure with support in $\mathrm{Green}(A).$

Define the \dfn{growth set} of $A$ to be $$\mathrm{Grow}(A)= \{x\notin A| \iota(\mathrm{topiary}(A))> 0\}.$$
\begin{observation}
    Let $B$ be compact.
    Suppose $\mathrm{Grow}(A)\cap B \neq \emptyset.$
    Then,
        $\mathrm{Grow}(A)\cap \mathrm{Topiary}(B) \neq \emptyset$
\end{observation}
Let $\mathrm{hedge}(A)$ be a finite signed measure with total variation $1$ such that $\iota(\mu)(x)\equiv 0.$
The  hedge need not be defined, we call sets $A$ where it is defined \dfn{prunable}.
Let $\mathrm{Prune}(A)$ be the support of the negative part of $\mathrm{hedge}(A).$
\begin{observation}
    Suppose $A$ is prunable.
    Suppose $\mathrm{Prune}(A) \neq \emptyset.$
    Then,
        $$\mathrm{Prune}(A)\cap \mathrm{Topiary}(A) \neq \mathrm{Prune}(A)$$
\end{observation}

The hedge and the topiary coincide exactly on topiaric frontiers. That is, if we know the topiaric frontier in advance, the problem is purely algebraic.
\begin{observation}
    $\mathrm{hedge}(\mathrm{Topiary}(A))=\mathrm{topiary}(A).$
\end{observation}

The \dfn{second greedy construction} of the topiary
$\textrm{sgreed}_n$ is constructed by taking the convex combination of $\textrm{sgreed}_{n-1}$  and a point mass at the maximizer of $\iota(\textrm{sgreed}_{n-1})$
such the aesthetic objective of that convex combination in maximized, and then eliminating any point masses in the distribution one by one that are now disadvangtageous and scaling the rest of the measure. 
Doing such minimizes the following chronic zig-zag-drag problem, which occurs when one chooses to greedily include elements which are not in the topiaric frontier.

\begin{figure}[h!]
\begin{tikzpicture}[scale=2, every node/.style={black}]

    \coordinate (A) at (-3,1);
    \coordinate (B) at (2,1);
    \coordinate (C) at (0,2);
    \coordinate (O) at (0,0);
    \coordinate (P) at (4/5,8/5);
    \coordinate (Q) at (-1.47*.155,1.47*.987);
    \coordinate (R) at (1.38*0.198366, 1.38*0.980128);

    \draw[thick, ->] (O) -- (C);
    \draw[dashed] (C) -- (B);
    \draw[thick,->] (O) -- (P);
    \draw[dashed] (A) -- (P);
    \draw[thick, ->] (O) -- (Q);
    \draw[dashed] (Q) -- (B);
    \draw[thick, ->] (O) -- (R);
    \draw[dashed, pink] (A) -- (B);

    \draw[thin, ->, color=brown] (O) -- (C);
    \draw[ ->, color=brown] (O) -- (B);
    \draw[->,color=brown] (O) -- (A);

    \draw[ draw=black,latex-latex] (C) node[below right] {\tiny $\mathrm{bad}_0$};
    \draw[draw=black,latex-latex] (P) node[below right] {\tiny $\mathrm{bad}_1$};
    \draw[draw=black,latex-latex] (Q) node[below right] {\tiny $\mathrm{bad}_2$};
    \draw[draw=black,latex-latex] (R) node[below right] {\tiny $\mathrm{bad}_3$};
    \draw[draw=black,latex-latex] (0,1) node[below right] {\tiny $topiary$};

\end{tikzpicture}
\caption{Take our reproducing kernel to correspond to the grammian of $(-3,1)$, $(0,2)$ and $(2,1)$ and assume $\psi=0.$ Starting from a bad seed and applying the greedy algorithm,
we see a chronic zig-zag-drag pattern. If we had chosen a proper intial state, the algorithm in fact halts in one step. Here, the aesthetic objective is merely the negative of the norm squared
over two.}
\end{figure}
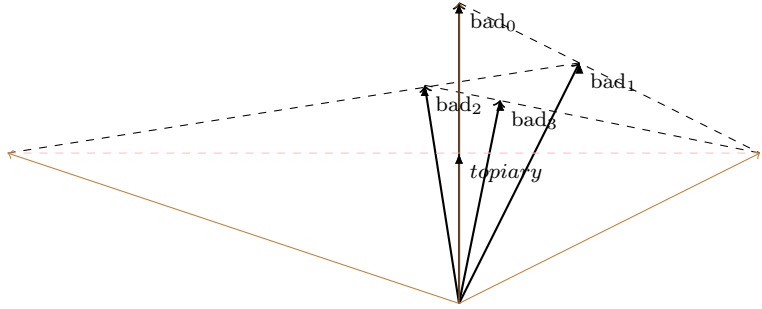

    \subsection{Algebraic approaches and combinatorics}
   
    One can consider the problem algebraically for finite sets as well. The key being when one wants to augment some proposed support $K$ with some new $x$
    while preserving the nullity of the aesthetic margin on its support, one may need to remove some elements from $K$ as they may be redundant or incompatiable.
    
    \begin{theorem}[Discussion theorem]
        Let $\mathcal{H}$ be real reproducing kernel Hilbert space on some domain $\Omega.$
        Let $\psi$ be a continuous function on $\Omega.$
        Let $K \subseteq \Omega$ be a finite topiaric index.
        Suppose $\topiary(K)$ has positive aesthetic margin at some $x.$
        There exists $B \subseteq K$ such that 
            $\mathfrak{O}(\topiary(B\cup\{x\}))> \mathfrak{O}(\topiary(K)).$
    \end{theorem}
    \begin{proof}
        Take $B$ to be the topiary of $K \cup \{ x\}.$
    \end{proof}

    The proof is hardly enlightening or effective. However, the following approach is:
    \begin{enumerate}
        \item Initialize $\mu$ to be $\topiary(K).$
        \begin{enumerate}
            \item Find a signed distribution $\nu$ such that the mass of $K$ is $1$ on the support of $\mu$ such that the embedded measures $\delta_x(y) = \nu(y)$ for $y$ in the support of $\mu.$
            \item Find the maximum $t$ such that $\mu + t(\delta_x-\nu)$ is a positive measure and $\iota(\mu + t(\delta_x-\nu))(x)\geq 0$
            \item If $\iota_{\mu + t(\delta_x-\nu)}(x) = 0,$ we are done. Otherwise, replace $\mu$ with $\mu + t(\delta_x-\nu)$ and repeat.
        \end{enumerate}
    \end{enumerate}
    However, we caution that such an approach may not produce the optimal $B,$ and we do not get the runtime guarantees given in the topiaric law of large numbers. If we remove $0$ or $1$ elements from
    $K$ to obtain $B$ indeed it must be optimal. Thus, in terms of adding mass at the optimal $x$ in a way that preserves the constancy of aesthetic margin on its support except perhaps at $x$,
    we see that there is a sort of ko rule as in the game of Go-- if we add $x$ to remove $y$, we can not immediately desire to add back $y.$
    Because of the lack of meaningful runtime estimates for such an algorithm, at least those apparent to the author, one fears there may be some metaphorical ko fights,
    where it adds $x$ and removes $y$ then adds some other point, then adds back $y$ and removes $x$, then adds some other point, then adds back $x$ to remove $y$ and so on.

    \begin{figure}
    \centering
     \animategraphics[width=10cm, loop, autoplay, poster=last]{3}
    {topiary-}
    {0}
    {302}
    \caption{The concrete version of updating via the discussion theorem, where $x$ is chosen to maximize the aesthetic margin. Here the space $\mathcal{H}$ is the real Hardy space on the disk. 
    The goldenrod curve represents the set $K$, and the yellow part is where the aesthetic margin is negative, the purple part where it is postive. The red points represent the current support. The pink point represents the most recent point we added. We have taken $\psi=0.$ Note that each component of the purple parts (and yellow parts for that matter) touches the boundary, and that $0$ is always in the purple part as it represents somewhat of an invisible index. We will discuss how the combination solves mazes later on. One may look at this as trying to ``positively classify" $K.$ In the animation, we see it start with a stereotype of $K$ which persists and is rarely covering
much of $K$ except the extremes. Slowly, it resolves more and works to cover gaps it has missed, sometimes making a large shift in its understanding.
We have neglected to depict the Julia-Caratheodory theorem as $\psi=0,$ so the Julia-Caratheodory theorem projects the goldenrod set to a line $x=-y$, under which the image of the support is the origin. The goldenrod set is projected to the ray the left the origin, and $0$ is projected to some point to the right of the origin.}
    \label{fig:my_label}
    \end{figure}

    \subsection{Post hoc no-backtracking  presentations}
    We now discuss the theory of dissection and construction of finite topiaric indices.
    \begin{theorem}[Finite accessibility]
        Let $\mathcal{H}$ be real reproducing kernel Hilbert space on some domain $\Omega.$
        Let $\psi$ be a continuous function on $\Omega.$
        Let $K \subseteq \Omega$ be a finite topiaric index. There exists an element $x$ such that
        $K \setminus \{x\}$ is a finite topiaric index.
    \end{theorem}
    \begin{proof}
        Consider the topiaric index $K_0$ which is a proper subset of $K$ and optimizes 
        $\mathfrak{O}(\topiary(K_0)).$ As $K$ is a topiaric index, the update inequality gives that there exists a measure with higher aesthetic objective supported on $K_0 \cup \{x\}$ for some $x \in K.$ Thus, the topiaric index of $K_0 \cup \{x\}$ has higher aesthetic objective and thus can not be a proper subset of $K$ and thus must be equal to $K.$
    \end{proof}

    As an immediate corollary, we see that a topiaric index is a union of nested topiaric indices gained by inductively adding elements. Importantly, this also means that one may also \emph{deconstruct} a topiaric index in a particular order piece by piece.
    \begin{corollary}[Finitary no removal constructability]
        Let $\mathcal{H}$ be real reproducing kernel Hilbert space on some domain $\Omega.$
        Let $\psi$ be a continuous function on $\Omega.$
        Let $K \subseteq \Omega$ be finite and a topiaric index.
        There is a well-ordering on $K$ such that each initial segment is a topiaric index.
    \end{corollary}
    The point being that if one knew the topiaric frontier in advance, one could come up with an explanation of it such that we start with one thing, add another which is compatiable repeatedly, namely needing no removals.
    Such a presentation makes the portfolio look predetermined, it is more expedient to be understood, and witnessed to be correct.

    It is unclear if there is an infinite analog of no removal constructability, which we give a formal problem for below.
    \begin{problem}[Existence of a topiaric branch]
        Let $\mathcal{H}$ be real reproducing kernel Hilbert space on some domain $\Omega.$
        Let $\psi$ be a continuous function on $\Omega.$
        Let $K \subseteq \Omega$ be compact.
        Does there exist a sequence points $(x_n)^\infty_{n=1}$ such that each $K_n = \{x_1,\ldots,x_n\}$
        is a topiaric index and $topiary(K_n) \rightarrow topiary(K)$ weakly.
    \end{problem}

    Moreover, one expects in ``real data," whatever that means, for there to outliers or elements with otherwise unreliable behavior, which somewhat explains why 
    models might fail, but suggests that one wants to thus find large topiaric indices within your space, without regard to true optimality.
    \begin{problem}[Broad index problem]
        Given a set $K$ find the largest topiaric index possible.
    \end{problem}

\section{Examples}
    \subsection{Minotaur in the hedge maze}

        Consider the real Fock space with kernel $A(z, w) = \Re e^{z,\overline{w}}.$
        Let $\psi =0.$
        Let $M\subseteq \mathbb{C}$ be a compact set with path-connected complement, and have $0\notin M$.
        Starting at $0$ following the gradient of the aesthetic objective of $\topiary(M)$ gives a path connecting
        $0$ to $\infty$ not intersecting $M.$ We note that if the boundary of $M$ is a smooth analytic curve, the topiaric index is finite,
        as any accumulation on a compact smooth closed real analytic curve would force the function to be globally constant as the aesthetic 
        margin of the topiary is constant exactly on the topiaric index. We also note that the contour of $0$ of the aesthetic objective touches
        $M$ exactly at the topiaric index of $M.$ If $M$ is the closure of its interior, indeed such a contour must be tangent.
        (Indeed, the topiary behaves somewhat like a support vector machine from the theory of machine learning and cybernetic intelligence.
        See \cite{poggio2003mathematics} for insight into support vector machines.)

        Of course, on any set winding around $0,$ we must have the topiary be constant. Thus, by the invisible index theorem,
        we are approximating the kernel of $0.$ Recursively defined mazes can encode universal computers. (In the sense that if there is
        path from some particular point,
        then some particular clause is true. That is, one can use a tree to write down the letters of a possible proof which deposits into a pool
        when it finds a legitimate proof.) There is at least some weak analogy of a market absent of shocks behaving like a diffusion process (and likely such has the stocks as waterwheels gathering the flux.) 

        One can do similarly over other spaces of harmonic functions in two variables
        or to reach other locales by minimizing $\|\mu-\delta_\alpha\|$ where $\alpha$ is some point of interest and one wants to find a path from 
        the origin to the boundary.

        \begin{figure}[H]
    \centering
     \animategraphics[width=10cm, loop, autoplay, poster=last]{3}
    {topiaron-}
    {0}
    {185}
    \caption{The harmonic conjugates of the approximating sequence coming from the concrete implementation of the discussion theorem. The goldenrod part is our ``maze" and plot is colored by distance from
    the value at $0.$ Due to conformality, (or equivalently satisfaction of the Cauchy-Riemann equations) the resulting bright curve through the origin represents a gradient descent path for the topiary.}
    \label{fig:my_label}
    \end{figure}

    It is helpful to constrast with the typical way we see mazes resolved physically in terms of static flows \cite{adamatzky2017physical}.
        Typically, these involve putting a large amount of stuff (such as water or charge, or even a growing slime mold, although its true processes are more complicated) through some maze which is fundamentally inoculated from the stuff.
        Eventually, because of some fundamental nature of this stuff, it finds an efficient path to flow through the maze. Fundamentally, these implement a breadth first search called Lee's algorithm as noted in \cite{adamatzky2017physical}.
        In some cases, one can observe optimal path heat or light up\cite{ayrinhac2018electron}, which is pretty
        fast essentially because Avogadro's number is so large.
        We lack any real insulation on our maze, and yet we obtain a similar gradient solving the maze--
        the invisible hand is caught in the cookie jar anyway. With the weak analogy that markets exist to diffuse value, one conjectures there may exist some analogous
        interpretation in terms of value flows. Analogies may also be made, and perhaps extended to deal with some shocks and stimuli within the control systems philosophy of
        biological and social systems, with the work of Doyle et al. \cite{doyle2011architecture}.

    \subsection{Trichotomy}
        Lets look at the image of the point mass at $0$ again in the real Fock space. Fix $M\subseteq \mathbb{C}$ compact with analytic boundary. Again let $\psi$ be $0.$
        There are three possibilities:
        \begin{enumerate}
            
            \item that $M$ winds around the origin, in which case one could take a measure supported on a continua of points,
            \item that we choose some finite collection of points from the boundary of $M,$
            \item that $0$ is in $M,$ thus we take a point mass at $0.$
        \end{enumerate}
        One can find caricatures of various ``value investing" theories in each, with the first essentially corresponding to constructing a broad index fund, the second corresponding to deep research and careful selection, the last
        being buying an index fund. Difference in strategy may be explained by difference in salience, which can be manifested as differences of $M.$ 
        In light of the invisible index theorem, they all approximate correspond to buying the invisible true index fund.

    \subsection{Topiaric portfoliation}
        Such an optimization problem arises naturally in portfolio theory,  as the median growth rate of a geometric Brownian motion is given by $\mu-\sigma^2/2$
where $\mu$ is the mean growth and $\sigma^2$ is the variance. (That is, the kernel matrix arises from some covariance matrix and the $\psi$ gives the performances.)
 Assuming the presence of risk-free assets such as cash or bonds of firm credit, one 
sees that \emph{an asset with maximum mean must be in the green frontier,} which helps avoid the chronic zig-zag-drag problem. In portfolio theory, the hedge corresponds to the Markowitz optimal portfolio
\cite{markowitz, sharpe}
for choice of risk tolerance parameter two, and other risk tolerances may be obtained by scaling the kernel. (Those correspond to the optimal portfolio keeping at least some fixed fraction in risk-free reserve, or playing with
at some fixed amount of margin.) 
Optimal portfolios with nonnegative weights are often supported on small sets, whereas those with arbitrary weights are diffuse 
\cite{green,best}, hence in practice, greedy computation
of the topiary is a low dimensional problem which is \emph{significantly easier and more robust.} Note also that large covariance matrices are often ill-conditioned, and thus the Markowitz optimal portfolio may be unknowable with much certainty.

    Given that large amounts of empirical data will contain ``lucky" and ``unlucky" outliers, one may want to mildly correct the data according to your \dfn{risk belief}. For example,
    if the return of a security exceeds its variance by more that the risk free rate, one would be incentivized to go all in on such a security, buying none of the risk free asset.
    One can correct such by either decreasing the mean, increasing the variance, or both and retain the positive semidefiniteness of the kernel. Such is fundamentally different from the Markowitzian notion
    of \dfn{risk tolerance} which is hard to reconcile with the growth of geometric Brownian motions.

\begin{problem}[Conservative data treament problems]
        Develop a theory of how solutions to quadratic programs behave when one takes $\psi$ is replaced by $\tilde{\psi}$
such that $\tilde{\psi}\leq \psi$ and the norm $\|\cdot\|_{\mathcal{H}}$ is replaced by $\|\cdot\|_{\tilde{\mathcal{H}}}$
such that $\|\cdot\|_{\mathcal{H}}\leq \|\cdot\|_{\mathcal{H}}.$
\end{problem}

    Moreover, since the optimal portfolio takes its maximum value on its support, assuming identical means, it must be almost paradoxically least correlated with the securities contained in the portfolio. 
In terms of portfolio theory, our median growth optimization is a bit different from more traditional Markowitz efficient frontier \cite{markowitz} 
or Sharpe ratio \cite{sharpe}
based methods. We view it as beneficial, although perhaps only heuristically,
that we take positive convex combinations of securities, as that avoids some extrapolatory error in terms of globally minimizing the objective over portfolios with arbitrary weights summing to $1$. From a finance point of view, it also avoids some of the fundamental problems with shorts and margin, although certainly the perspective on optimization here might shed some insight there, especially when hidden fees are taken into account.
Taking into account the theory laid out above gives reasonably fast performance. (Informally, we will comment if one does not, one will observe extremely slow convergence due to chronic zig-zag-drag.) 

Note too, that if we change the method via which return and covariance are obtained, the topiary will change.
Thus too, if the currency changes, the topiaric portfolio will change.
Moreover, what it takes for one to raise their score may be what another needs to depart for a rise. Thus, certain mutualism exists due to uneven opportunity, resolution, or information. Additionally, as the topiaric frontier is sparse, one expects there may be many locations to compete in-- the game is massively multiplayer.

        \begin{figure}[H]
    \centering
     \animategraphics[width=12cm, loop, autoplay, poster=last]{10}
    {out}
    {1}
    {252}
    \caption{An animation of our Julia-Cartheodory theorem for real securities data. Data is provided without warranty of correctness, and the labels have been anonymized, and is meant as a demonstration of the theorem. Specifically,
    it neither represents financial advice nor is it intended for trading purposes.
    However, we would encourage further empirical investigation. The green dot represents the optimal portfolio, and the (often obscured) red point is fixed, representing an instrument with fixed yield. 
    Here the excess marginal return is the aesthetic margin and the covariant risk is the covariance of each security with the portfolio. }
    \end{figure}

    \subsection{Around organizational structure}
          Operations research problems that organizations must (at least implicitly) solve are analogous to portfolio optimization.
           These generally deal with how organizations, such as businesses, governments, and so on, allocate resources.

        Organizations must commit resources to unreliable sources in order to function. For example, in a foraging
        type task, one may want to visit multiple areas to look for a desired resource, but not look at areas for which there is little or no utility.

          The ``going well together" analogy works here, and informs what suborganizations can survive on their own.
          Suborganizations coming from one division may not be able to thrive.

             \begin{figure}[H]
  		\includegraphics[width=100pt]{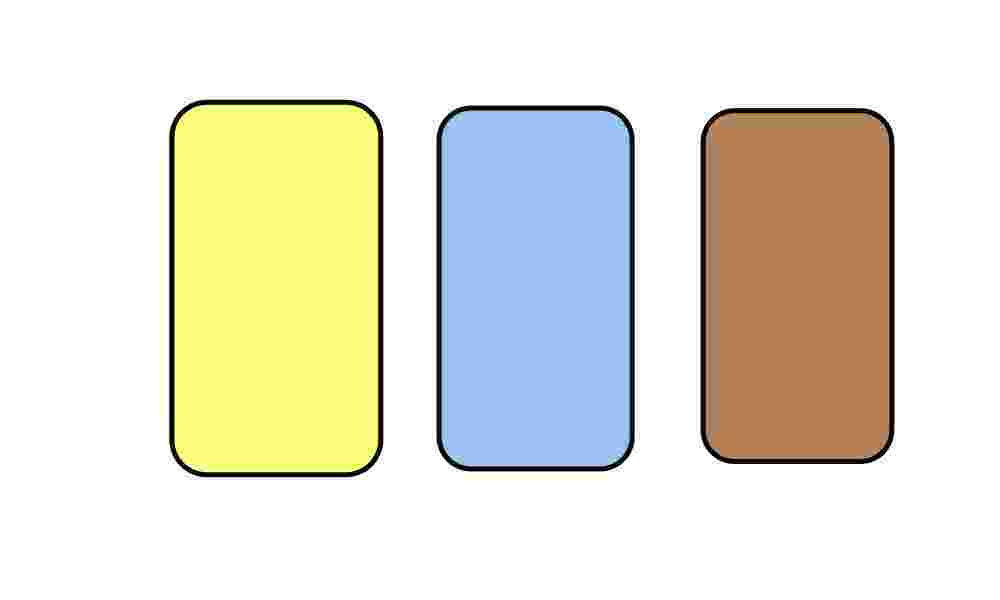}
        \includegraphics[width=100pt]{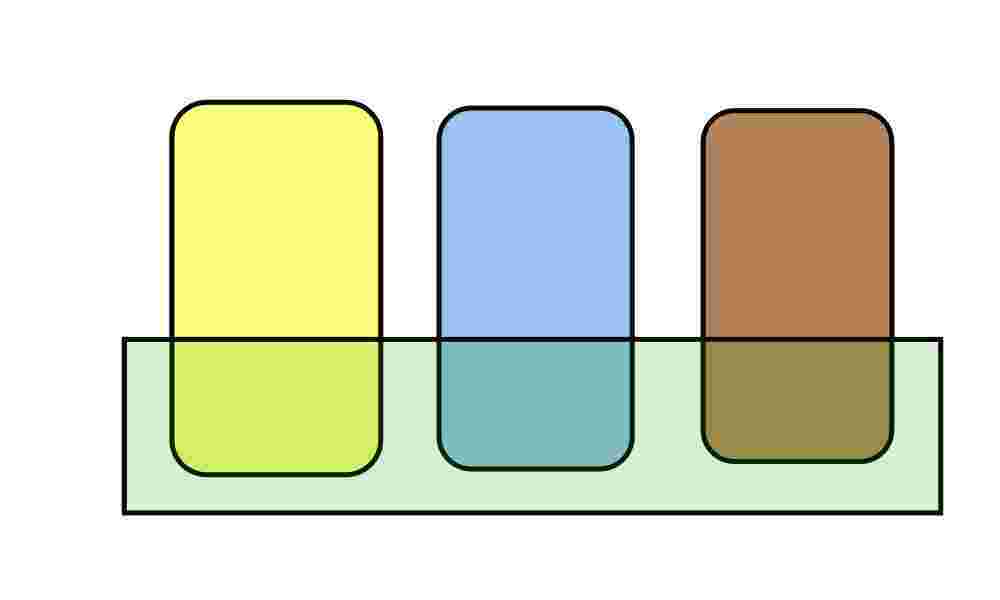}
        \includegraphics[width=100pt]{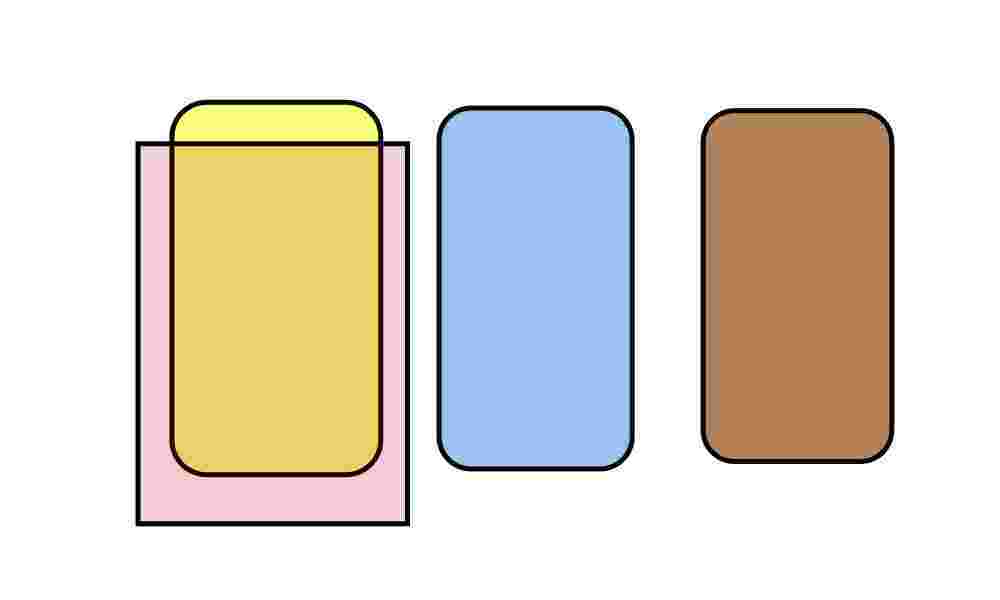}
        \caption{Lateral suborganizations can be better than subdivisions such as manufacturing, advertising and so on as they piggyback off the risk reduction of the ambient organization.
        That is, groups of parallel vertically integrated business pathways likely have more of a chance of surviving on their own.}
	    \end{figure}

        \begin{figure}[H]
  		\includegraphics[width=100pt]{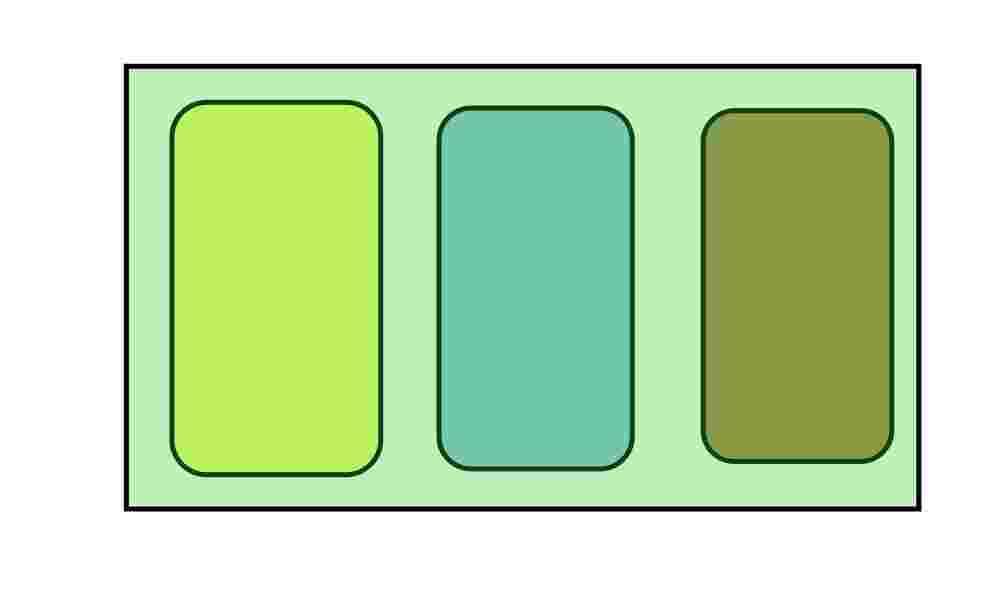}
        \includegraphics[width=100pt]{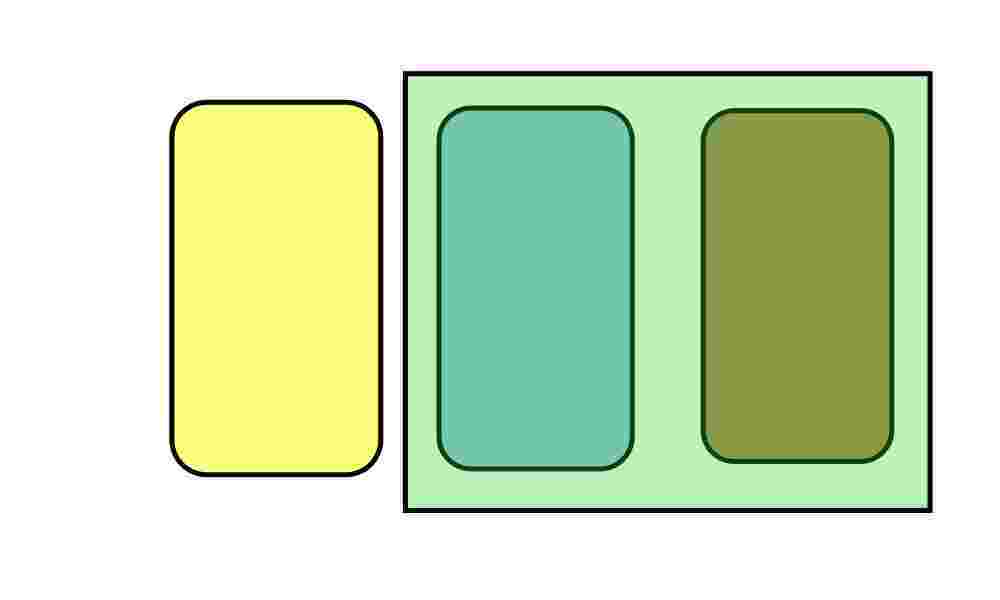}
        \includegraphics[width=100pt]{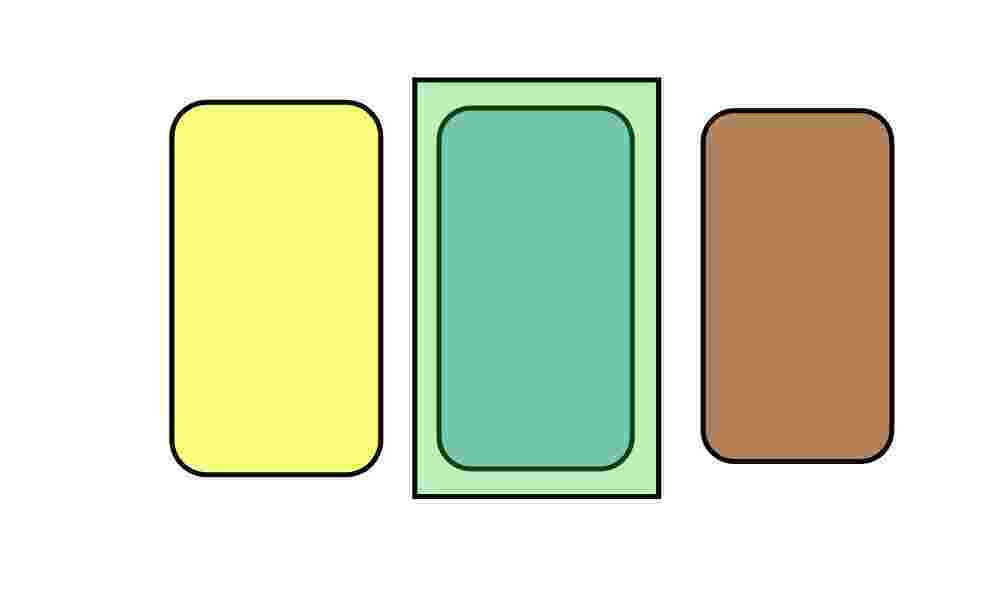}
        \caption{The no removal finite constructability theorem suggests one may posess as pathway to divest from individual parts as long as one does it in a particular order.
        Note, however, it may be necessary to do some reallocation in the remainder, and fitness still decreases as we remove parts.}
	    \end{figure}

        We see some heuristic justification for why businesses do not diversify more. For example, they end up in an organization that can not grow except by needing to cut something first,
        even if the ultimate goal would be broad diversification. Note that a business entity faces significantly more reallocation costs than an investor in liquid equities such as stocks.

    \subsection{The adaptive market hypothesis}
        In terms of adaptive market hypotheses \cite{andrewlo}, we see the problem as maxmimizing $$\int \psi \dd \mu - \|\mu - \nu \|^2/2$$
        where $\nu$ is a distribution on a set $K_1$ representing the present and $\mu$ is a distribution on $K_2$ representing the past.
        That is, one can optimize determinacy and return similarly, although you must optimize over two measures instead of one.
        The point being that we expect allocations which are similar to superpositions of allocations of different securities from the past under our kernel embedding to have similar returns.


\bibliography{references}
\bibliographystyle{plain}


\end{document}